\documentclass[11pt,a4paper]{article}
\usepackage{amsmath}   
\usepackage{amssymb}   
\usepackage{amsthm}    
\usepackage{stmaryrd}  
\usepackage{titletoc}  
\usepackage{mathrsfs}  
\usepackage{enumitem}
\usepackage{fullpage}

\usepackage{graphicx}
\usepackage{url}

\theoremstyle{plain}
\newtheorem{theorem}{Theorem}[section]
\newtheorem{proposition}[theorem]{Proposition}
\newtheorem{corollary}[theorem]{Corollary}
\newtheorem{lemma}[theorem]{Lemma}
\theoremstyle{remark}
\newtheorem{remark}[theorem]{Remark}
\newtheorem{example}[theorem]{Example}
\theoremstyle{definition}
\newtheorem{assumption}[theorem]{Assumption}
\newtheorem{definition}[theorem]{Definition}

\usepackage{parskip}
\makeatletter
\def\thm@space@setup{%
  \thm@preskip=\parskip \thm@postskip=0pt
}
\makeatother

\numberwithin{equation}{section}

\DeclareMathOperator*\essinf{ess\,inf}

\DeclareMathOperator*\arctanh{arc\,tanh}

\begin{document}

\title{Optimal Trade Execution with Instantaneous Price Impact \\and Stochastic Resilience\footnote{Financial support through the \textit{CRC 649 Economic Risk} and \textit{d-fine GmbH} is gratefully acknowledged. We thank Peter Bank, R\"udiger Frey, Nizar Touzi and seminar participants at various institutions for valuable comments and feedback.}}

\author{Paulwin Graewe\footnote{Department of Mathematics, Humboldt-Universit\"{a}t zu Berlin, Unter den Linden 6, 10099 Berlin, Germany, \texttt{graewe@math.hu-berlin.de}.} \and Ulrich Horst\footnote{Department of Mathematics and School of Business and Economics, Humboldt-Universit\"{a}t zu Berlin,
Unter den Linden 6, 10099 Berlin, Germany, \texttt{horst@math.hu-berlin.de}.}}
\maketitle

\begin{abstract}
	We study an optimal execution problem in illiquid markets with both instantaneous and persistent price impact and stochastic resilience when only absolutely continuous trading strategies are admissible. In our model the value function can be described by a three-dimensional system of backward stochastic differential equations (BSDE) with a singular terminal condition in one component. We prove existence and uniqueness of a solution to the BSDE system and characterize both the value function and the optimal strategy in terms of the unique solution to the BSDE system. Our existence proof is based on an asymptotic expansion of the BSDE system at the terminal time that allows us to express the system in terms of a equivalent system with finite terminal value but singular driver. 
\end{abstract}



\section{Introduction and overview}

	Let $T \in (0,\infty)$. Let $(\Omega,\mathcal F,(\mathcal F_t)_{t\in[0,T]},\mathbb P)$ be filtered probability space that carries a $m$-di\-men\-sion\-al standard Brownian motion  $W=(W_t)_{t\in[0,T]}$. We assume throughout that $(\mathcal F_t)_{t\in[0,T]}$ is the filtration generated by $W$ completed by all the null sets and that $\mathcal F = \mathcal F_T$. We denote by $L^\infty_{\mathcal F}(0,T;\mathbb R^d)$ and $L^\infty_{\mathcal F}(\Omega; C([0,T];\mathbb R^d))$, respectively, the set of progressively measurable $\mathbb R^d$-valued, respectively, continuous processes that are essentially bounded. $L^2_{\mathcal F}(0,T;\mathbb R^d)$ denotes the set of progressively measurable $\mathbb R^d$-valued processes $(Y_t)_{t\in[0,T]}$ such that $E[\int_0^T|Y_t|^2\,dt]<\infty$, and $L^2_{\mathcal F}(\Omega;C([0,T];\mathbb R^d))$ denotes the subset of all such processes with continuous sample paths such that $E[\sup_{t\in[0,T]}|Y_t|^2]<\infty$. All equations and inequalities are to be understood in the $\mathbb P$-a.s.~sense.

	In this paper we address the linear-quadratic non-Markovian stochastic control problem 
\begin{equation} \label{control-problem}
	\essinf_{\xi \in L^2_{\mathcal F}(0,T;\mathbb R)} \mathbb E\left[ \int_0^T\{\tfrac{1}{2}\eta\xi_s^2+\xi_sY_s+\tfrac{1}{2}\lambda_sX_s^2\}\,ds		\right]
\end{equation}
subject to 
\begin{equation*} 
\left\{\begin{aligned}
	X_t&= x - \int_0^t \xi_s\,ds,  \quad t \in [0,T],\\
	X_T &= 0, \\
	Y_t&=y + \int_0^t\{-\rho_sY_s+\gamma\xi_s\}\,ds, \quad t \in [0,T].
\end{aligned}\right.
\end{equation*}
Here, $\eta$ and $\gamma$ are positive constants and $\rho$ and $\lambda$ are progressively measurable, non-negative and essentially bounded stochastic processes:
\[
	\eta>0, \gamma \in \mathbb R_+; \quad \rho, \lambda \in L^\infty_{\mathcal F}(0,T;\mathbb R_+).
\]
The process $(X_t,Y_t)_{t \in [0,T]}$ is called the \textit{state process}. It is governed by the \textit{control} $\xi=(\xi_t)_{t \in [0,T]}$. The processes $\lambda=(\lambda_t)_{t \in [0,T]}$ and $\rho=(\rho_t)_{t \in [0,T]}$ are uncontrolled. 
Control problems of the above form arise in models of optimal portfolio liquidation under market impact with stochastic resilience. In such models $X_t \geq 0$ denotes the number of shares an investor needs to sell at time $t \in [0,T]$, $\xi_t$ denotes the rate at which the stock is traded at time $t \in [0,T]$, and the terminal state constraint $X_T = 0$ is the \textit{liquidation constraint}. The process $Y$ describes the \textit{persistent price impact} caused by past trades in a block-shaped limit order book market with constant order book depth $1/\gamma>0$ as in Obizhaeva and Wang~\cite{ObizhaevaWang13}. One interpretation is that the trading rate $\xi$ adds a drift to an underlying fundamental martingale price process. This results in an execution price process of the form 
\[
	\tilde S_t = S_t - \eta \xi_t - Y_t
\]
where $S_t$ denotes the underlying fundamental martingale price process. The process $\rho \in L^\infty_{\mathcal F}(0,T;\mathbb R_+)$ describes the rates at which the order book recovers from past trades. The constant $\eta > 0$ describes an additional \textit{instantaneous impact factor} as in Almgren and Chriss~\cite{AlmgrenChriss00}. The first two terms of running cost term in \eqref{control-problem} capture the expected liquidity cost resulting from the instantaneous and the persistent impact, respectively. The third term can be interpreted as a measure of the market risk associated with an open position. It penalizes slow liquidation. We allow the risk factor $\lambda$ to be stochastic. 

	The majority of the optimal trade execution literature allows for only one of the two possible price impacts. The first approach, initiated by Bertsimas and Lo~\cite{BertsimasLo98} and Almgren and Chriss~\cite{AlmgrenChriss00}, describes the price impact as a purely temporary effect that depends only on the present trading rate and does not influence future prices. The impact is typically assumed to be linear in the trading rate, leading to a quadratic cost term of the form $\frac{1}{2} \eta \xi_t^2$. 
	
	In our framework,
the special case $\rho \equiv 0$, $\gamma=0$, $y=0$, and $\lambda \equiv const$ corresponds to model of Almgren and Chriss~\cite{AlmgrenChriss00}. Their model has been extended by many authors. Closest to our work are the papers by Ankirchner et al.~\cite{AnkirchnerJeanblancKruse14}, Graewe et al. \cite{GraeweHorstQiu13}, Horst et al.~\cite{HorstQiuZhang16}, and Kruse and Popier~\cite{KrusePopier16}. They all consider non-Markovian liquidation problems with purely temporary price impact where the cost functional is driven by general adapted factor processes and where the HJB equation can be solved in terms of {\it one-dimensional} BSDEs or BSPDEs with singular terminal values, depending on the dynamics of the factor processes. A general class of Markovian liquidation problems has been solved in Schied~\cite{Schied13} by means of Dawson--Watanabe superprocess. This approach avoids the use of HJB equations and uses instead a probabilistic verification argument based on log-Laplace functionals of superprocesses. 

	A second approach, initiated by Obizhaeva and Wang~\cite{ObizhaevaWang13} assumes that price impact is persistent with the impact of past trades on current prices decaying over time. When impact is persistent one often allows for both absolutely continuous and singular trading strategies. In~\cite{ObizhaevaWang13} the authors assumed constant resilience and market depth. Fruth et al.~\cite{FruthSchoenebornUrusov14} generalized the model to deterministic time-varying market depths and resiliences and obtained a closed form solution by calculus of variation techniques. In the follow up work~\cite{FruthSchoenebornUrusov15} the authors allowed for stochastic liquidity parameters. They showed the state space divides into a trade and a no-trade region but did not obtain an explicit description of the boundary. Characterization of optimal strategies results in terms of coupled BSDE systems were obtained by Horst and Naujokat \cite{HorstNaujokat14} for a model of optimal curve following in a two-sided limit order book. An explicit solution of the related free-boundary problem in a model with infinite time horizon and multiplicative price impact has recently been given by Becherer et al.~\cite{BechererBilarevFrentrup16}. 

	In this paper we analyze a stochastic control problem arising in models of optimal trade execution with both instantaneous and persistent price impact where only absolutely continuous trading strategies are admissible. Economically, the restriction to absolutely continuous strategies means that the instantaneous impact is the dominating factor. Mathematically, it allows us to formulate the resulting control problem within in a classical, rather than singular stochastic control framework, and to obtain a closed form solution for both, the value function and the optimal trading strategy. Characterizing the value function is typically hard if singular controls are allowed. In fact, when both absolutely continuous and singular controls are admissible as in e.g.~\cite{HorstNaujokat14}, one typically only obtains characterization results for optimal controls using maximum principles. 

	Within our modeling framework,  the value function can be represented in terms of the solution to a fully coupled \textit{three-dimensional} stochastic Riccati equation (BSDE system). For the benchmark case of constant model parameters the stochastic system reduces to a deterministic ODE system. For this case we illustrate how our model can be used to approximate liquidation models with block trades and can, hence, be viewed as a first step towards a unified approach to singular and regular stochastic control problems with singular terminal values.     

While uncoupled ODE, respectively, BSDE systems arise in the trade execution models of Gatheral and Schied~\cite{GatheralSchied11} or Kratz~\cite{Kratz14}, respectively, Ankirchner and Kruse~\cite{AnkirchnerKruse15}, our model seems to be the first that requires the analysis of multi-dimensional BSDE systems.  In proving the existence of a unique solution to the BSDE system that describes the value function two challenges need to be overcome. First,  the liquidation constraint imposes a singular terminal condition on the first component of the BSDE system. Second, our BSDE system does not satisfy the quasi-monotonicity condition that is necessary for the multi-dimensional comparison principle in \cite{HuPeng06} to hold. In a one-dimensional setting BS(P)DEs with singular terminal values are well understood and an array of existence of solution results and comparison principles has been obtained in the literature. The majority of the existing results including \cite{AnkirchnerJeanblancKruse14,GraeweHorstQiu13,HorstQiuZhang16,KrusePopier16} rely on a finite approximation of the singular terminal value. The (minimal) solution with singular terminal value is then obtained by a monotone limit argument. 

	We extend the asymptotic expansion approach introduced in Graewe~et~al.~\cite{GraeweHorstSere13} to BSDE systems. The idea is to determine the precise asymptotic behavior of a potential solution to the BSDE system at the terminal time by finding appropriate a priori estimates. The asymptotics of the solution at the terminal time allows us to characterize the solution to the BSDE system with singular terminal value in terms of a BSDE with finite terminal value yet singular driver, for which the existence of a solution in a suitable space can be proved using standard fixed point arguments. Finally, we establish the verification result from which we deduce uniqueness of solutions to the BSDE system as well as a closed-form representation of the optimal trading strategy.

	Establishing the a priori estimates for our BSDE system is key for both the proof of existence of a solution and the verification theorem. As pointed out above the BSDE system that characterizes the value does not satisfy the quasi-monotonicity condition of Hu and Peng~\cite{HuPeng06}. In order to overcome this problem we consider the joint dynamics of the BSDE that describes the value function and two additional BSDEs that describe the candidate optimal trading strategy. Using the comparison principle for BSDE systems in~\cite{HuPeng06} we first determine the range of all these processes from which we then deduce the desired deterministic upper bounds for the coefficients of the value function. 
	
	The remainder of this paper is structured as follows. The stochastic control problem is formulated in Section~\ref{sec-problem-formulation}. The a priori estimates and asymptotic behavior of the solution is established in Section~\ref{sec-a-priori}. Existence to the HJB equation is proven in Section~\ref{sec-existence}. The verification argument is carried out in Section~\ref{sec-verification}. In Appendix~\ref{sec-appendix} we recall the multi-dimensional comparison principle for BSDEs and formulate a local $L^\infty$-existence result for BSDEs with locally Lipschitz drivers.
	
\textit{Notational convention.}  Whenever the notation $T^-$ appears  we mean that the statement holds for all the $T'<T$ when $T^-$ is replaced by $T'$, e.g., $L^2_{\mathcal F}(0,T^-;\mathbb R^{d\times m})=\bigcap_{T'<T}L^2_{\mathcal F}(0,T';\mathbb R^{d\times m})$. Furthermore, for $Y\in L_{\mathcal F}^\infty(\Omega,C([0,T^-];\mathbb R))$ we mean by $L^\infty$-$\lim_{t\rightarrow T}Y_t= \infty$ that for every $C>0$ there exists $T'<T$ such that $Y_t\geq C$ for all $t\in[T',T)$, $\mathbb P$-a.s.

\section{Main result} \label{sec-problem-formulation}

	For any initial state $(t,x,y)\in[0,T)\times\mathbb R\times\mathbb R$ we define by 
\begin{equation} \label{CP} 
	V_t(x,y):=\essinf_{\xi\in\mathcal A(t,x)} \mathbb E\left[\left.\int_t^T\{\tfrac{1}{2}\eta\xi_s^2+\xi_sY_s+\tfrac{1}{2}\lambda_sX_s^2\}\,ds\,\right|\mathcal F_t\right]
\end{equation}
the \textit{value function} of the stochastic control problem \eqref{control-problem} with respect to the state dynamics 
\[
\left\{\begin{aligned}
	dX_s&=-\xi_s\,ds, && s\in[t,T], && X_t=x,\\
	dY_s&=\{-\rho_sY_s+\gamma\xi_s\}\,ds, && s\in[t,T], &&Y_t=y,
\end{aligned}\right.
\] 
where only those \textit{controls} or (\textit{trading}) \textit{strategies} $\xi\in L_{\mathcal F}^2(t,T;\mathbb R)$ belong to the class $\mathcal A(t,x)$ of \textit{admissible} controls  that satisfy the terminal state constraint 
\[
	X_T=0 \quad \text{a.s.}
\]

\begin{assumption}
	We assume throughout that the coefficients to the control problem satisfy
\[ 
	\eta>0, \gamma \in \mathbb R_+; \quad \rho, \lambda \in L^\infty_{\mathcal F}(0,T;\mathbb R_+).
\]
\end{assumption}

\begin{remark}
	Notice that $X,Y\in L^2_{\mathcal F}(\Omega;C([t,T];\mathbb R))$ for any (admissible) control as $\xi\in L^2_{\mathcal F}(t,T;\mathbb R)$ and  $\rho\in L^\infty_{\mathcal F}(0,T;\mathbb R_+).$
\end{remark}

We solve the control problem by solving the corresponding stochastic Hamilton-Jacobi-Bellman (HJB) equation. Stochastic HJB equations for non-Markovian control problems were first introduced by Peng~\cite{Peng92}. In our model the stochastic HJB equation is given by the first-order stochastic partial differential equation,
\begin{equation} \label{HJB} 
	-dV_t(x,y)=\inf_{\xi\in\mathbb R}\{-\xi\partial_x V_t(x,y)-(\rho_t y-\gamma\xi)\partial_y V_t(x,y)+\tfrac{1}{2}\eta\xi^2+\xi y+\tfrac{1}{2}\lambda_t x^2\}\,dt-Z_t(x,y)\,dW_t.
\end{equation}

\begin{definition}	
	A pair of random fields $(V,Z):\Omega\times[0,T)\times \mathbb R\times\mathbb R\rightarrow \mathbb R\times\mathbb R^m$ is called a \textit{classical solution} to the above equation if it satisfies the following conditions:
\begin{itemize}
	\item for each $t\in[0,T)$, $V_t(x,y)$ is continuously differentiable in $x$ and $y$,
	\item for each $(x,y)\in\mathbb R^2$, $(V_t(x,y),\partial_x V_t(x,y),\partial_y V_t(x,y))_{t\in[0,T)}\in L_{\mathcal F}^\infty(\Omega;C([0,T^-];\mathbb R^3))$,
	\item for each $(x,y)\in\mathbb R^2$, $(Z_t(x,y))_{t\in[0,T)}\in L^2_{\mathcal F}(0,T^-;\mathbb R^m)$,
	\item for all $0\leq t\leq s<T$ and $x,y\in\mathbb R$ it holds that
\begin{align*}
	 V_t(x,y) &=  V_s(x,y)+\int_t^s\inf_{\xi\in\mathbb R}\{-\xi\partial_x V_r(x,y)-(\rho_r y-\gamma\xi)\partial_y V_r(x,y)+\tfrac{1}{2}\eta\xi^2+\xi y+\tfrac{1}{2}\lambda_r x^2\}\,dr\\
	 &\quad -\int_t^sZ_r(x,y)\,dW_r.
\end{align*}
\end{itemize}
\end{definition}

We prove the existence of a unique classical solution to the equation \eqref{HJB} and show that the value function is given by the random field $V$. 
The linear-quadratic structure of the control problem suggest the ansatz 
\begin{equation} \label{LQ-ansatz} 
\begin{split}
	V_t(x,y) &=\tfrac{1}{2}A_tx^2+B_txy+\tfrac{1}{2}C_ty^2 \\
	Z_t(x,y) &=\tfrac{1}{2}Z_t^Ax^2+Z_t^Bxy+\tfrac{1}{2}Z_t^Cy^2
\end{split}
\end{equation}
for the solution $(V,Z)$ to the HJB equation. The following lemma shows that this ansatz reduces our HJB equation to the following three-dimensional stochastic Riccati equation: 
\begin{equation}  \label{BSDE-infty} 
\left\{
\begin{aligned}
-dA_t&=\left\{\lambda_t-\eta^{-1} (A_t-\gamma B_t)^2\right\}dt-Z_t^A\,dW_t\\
-dB_t&=\left\{-\rho_t B_t+\eta^{-1}(\gamma C_t-B_t+1)(A_t-\gamma B_t)\right\}dt-Z_t^B\,dW_t\\
-dC_t&= \left\{-2\rho_t C_t-\eta^{-1}(\gamma C_t-B_t+1)^2\right\}dt-Z_t^C\,dW_t.
\end{aligned}
\right.
\end{equation}

\begin{lemma}
	If the vector $$\left((A,B,C),(Z^A,Z^B,Z^C) \right)\in L_{\mathcal F}^\infty(\Omega;C([0,T^-];\mathbb R^3))\times L_{\mathcal F}^2(0,T^-;\mathbb R^{3\times m})$$ solves the BSDE system~\eqref{BSDE-infty}, then the random field $(V,Z)$ given by the linear-quadratic ansatz~\eqref{LQ-ansatz} is a classical solution to the HJB equation~\eqref{HJB} such that the infimum in~\eqref{HJB} is attained by
\begin{equation} \label{feedback-form}
	\xi^*_t(x,y)=\eta^{-1}(A_t-\gamma B_t)x-\eta^{-1}(\gamma C_t-B_t+1)y.
\end{equation}
\end{lemma}
\begin{proof}	
	Let us fix $(t,x,y)\in[0,T)\times\mathbb R\times\mathbb R$. The Hamiltonian 
\begin{align*}
	h(\xi)&=-\xi\partial_x V_t(x,y)-(\rho_t y-\gamma\xi)\partial_y V_t(t,x)+\tfrac{1}{2}\eta\xi^2+\xi y+\tfrac{1}{2}\lambda_tx^2\\
	&=\tfrac{1}{2}\eta^{-1}\left(\eta\xi-\partial_x V_t(x,y)+\gamma\partial_y V_t(t,x)+y\right)^2-\tfrac{1}{2}\eta^{-1}\left(\partial_xV_t(x,y)-\gamma\partial_yV_t(x,y)-y\right)^2\\
	&\quad-\rho_ty\partial_y V_t(x,y)+\tfrac{1}{2}\lambda_t x^2
\end{align*}
is minimized at 
\[
	\xi^*=\eta^{-1}(\partial_x V_t(x,y)-\gamma\partial_y V_t(t,x)-y).
\]
In terms of the linear-quadratic ansatz \eqref{LQ-ansatz}, we obtain \eqref{feedback-form} and
\begin{align*}
	h(\xi^*)&= -\tfrac{1}{2}\eta^{-1}((A_t-\gamma B_t)x-(\gamma C_t-B_t+1)y)^2-\rho_ty(B_tx+C_ty)+\tfrac{1}{2}\lambda_t x^2\\
	&= \tfrac{1}{2}(\lambda_t-\eta^{-1}(A_t-\gamma B_t)^2)x^2+(-\rho_tB_t+\eta^{-1}(A_t-\gamma B_t)(\gamma C_t-B_t+1))xy\\
	&\quad+\tfrac{1}{2}(-2\rho_t C_t-\eta^{-1}(\gamma C_t-B_t+1)^2)y^2. \qedhere
\end{align*}
\end{proof}

In order to guarantee the uniqueness of a solution to the HJB equation we need to impose a suitable terminal condition. Due to the terminal state constraint $X_T=0$ we expect the trading rate $\xi$ to tend to infinity for any non-trivial initial position as $t \to T$. 
%
	We further expect the resulting trading cost to dominate any resilience effect. As a result, we expect that  
\[
	V_t(x,y) \sim V^{\rho=0}_t(x,y) \quad \mbox{as } t \to T
\]	
where $V^{\rho=0}_t(x,y)$ denotes the value function corresponding to the control problem with $\rho\equiv0$. If $\rho\equiv0$, then $Y^{\rho=0}=y+\gamma (x-X)$ and
\[
	\int_t^T\xi_sY_s^{\rho=0}\,ds=xy+\tfrac{1}{2}\gamma x^2,
\]
independently of the strategy $\xi\in\mathcal A(t,x)$. Hence,
\begin{align*}
V^{\rho=0}_t(x,y)&=\essinf_{\xi\in\mathcal A(t,x)} \mathbb E\left[\left.\int_t^T\{\tfrac{1}{2}\eta\xi_s^2+\tfrac{1}{2}\lambda_sX_s^2\}\,ds\right|\mathcal F_t\right]+xy+\tfrac{1}{2}\gamma x^2\\
	&=\tfrac{1}{2}(\tilde A_t+\gamma)x^2+xy,
\end{align*}
where $\tilde A$ is characterized in \cite{AnkirchnerJeanblancKruse14,GraeweHorstSere13} as the unique solution to the BSDE with singular terminal value 
\[
	\left\{
		\begin{aligned}	
		-d\tilde A_t &=\{\lambda_t-\eta^{-1}\tilde A_t^2\}\,dt-Z_t^{\tilde A}\,dW_t, \\
		\lim_{t \to T}\tilde A_t& = \infty \quad \text{in } L^\infty. 
		\end{aligned} 
	\right. 
\]
We therefore expect the coefficients of the linear-quadratic ansatz~\eqref{LQ-ansatz} to satisfy 
\begin{equation} \label{singular-terminal-condition}
	(A_t,B_t,C_t)\longrightarrow (\infty,1,0) \quad \text{in $L^\infty$ as $t\rightarrow T$.}
\end{equation}

The next theorem establishes an existence of solutions result for the BSDE system \eqref{BSDE-infty} when imposed with the singular terminal condition \eqref{singular-terminal-condition}. The proof is given in Section~\ref{sec-existence}. It is based on a multi-dimensional generalization of the asymptotic expansion approached introduced in~\cite{GraeweHorstSere13}.

\begin{theorem} \label{thm-existence}
	The BSDE system~\eqref{BSDE-infty} imposed with the singular terminal condition~\eqref{singular-terminal-condition} admits at least one solution $$((A,B,C),(Z^A,Z^B,Z^C))\in L_{\mathcal F}^\infty(\Omega;C([0,T^-];\mathbb R^3))\times L_{\mathcal F}^2(0,T^-;\mathbb R^{3\times m}).$$
\end{theorem}

The next theorem verifies the preceding heuristics; its proof is given in Section \ref{sec-verification}. In particular, it states that the value function is indeed of the form \eqref{LQ-ansatz}. As a result, there exists at most one solution to the BSDE system~\eqref{BSDE-infty} that satisfies~\eqref{singular-terminal-condition}. 

\begin{theorem} \label{thm-verification}
	Let  $((A,B,C),(Z^A,Z^B,Z^C))\in L_{\mathcal F}^\infty(\Omega;C([0,T^-];\mathbb R^3))\times L_{\mathcal F}^2(0,T^-;\mathbb R^{3\times m})$ be a solution to the BSDE system~\eqref{BSDE-infty} that satisfies the singular terminal condition~\eqref{singular-terminal-condition}. Then, the value function is of the linear-quadratic form \eqref{LQ-ansatz} and the optimal control is given by the feedback form \eqref{feedback-form}. In particular, the system admits at most one solution that satisfies~\eqref{singular-terminal-condition}.  
\end{theorem}
%
%
%
%

\par\bigskip

\begin{example}
In a deterministic benchmark model with a risk neutral investor $(\lambda \equiv 0)$ and constant deterministic resilience $(\rho_t \equiv \rho > 0)$ the above BSDE system reduces to the following ODE system:
\begin{equation*}
\left\{
\begin{aligned}
-\dot A_t&=-\eta^{-1}(A_t-\gamma B_t)^2, && 0\leq t<T;\qquad \lim_{t\rightarrow T} A_t=+\infty;\\
-\dot B_t&=-\rho B_t+\eta^{-1}(\gamma C_t-B_t+1)(A_t-\gamma B_t), && 0\leq t<T;\qquad \lim_{t\rightarrow T} B_t=1;\\
-\dot C_t&=-2\rho C_t-\eta^{-1}(\gamma C_t-B_t)^2, && 0\leq t<T;\qquad \lim_{t\rightarrow T} C_t=0.
\end{aligned}
\right.
\end{equation*}
Using the asymptotic expansion introduced in \eqref{asymptotic-ansatz} the above ODE system can be solved by solving the corresponding ODE system \eqref{BSDE_0}. That ODE system has finite terminal values yet singular nonlinearity. It can be solved numerically using the MATLAB package \texttt{bvpsuite}~\cite{Kitzhoferetal10}. This package is designed for solving ODE systems with regular singular points. The optimal trading strategies for different choices of the instantaneous impact factor $\eta$ and the optimal trading strategies of the benchmark models by Almgren and Chriss~\cite{AlmgrenChriss00} and Obizhaeva and Wang~\cite{ObizhaevaWang13} are depicted in Figure~1. 
\begin{figure}[!htb]\label{figure}
	\centering
	\includegraphics[width=0.7\textwidth]{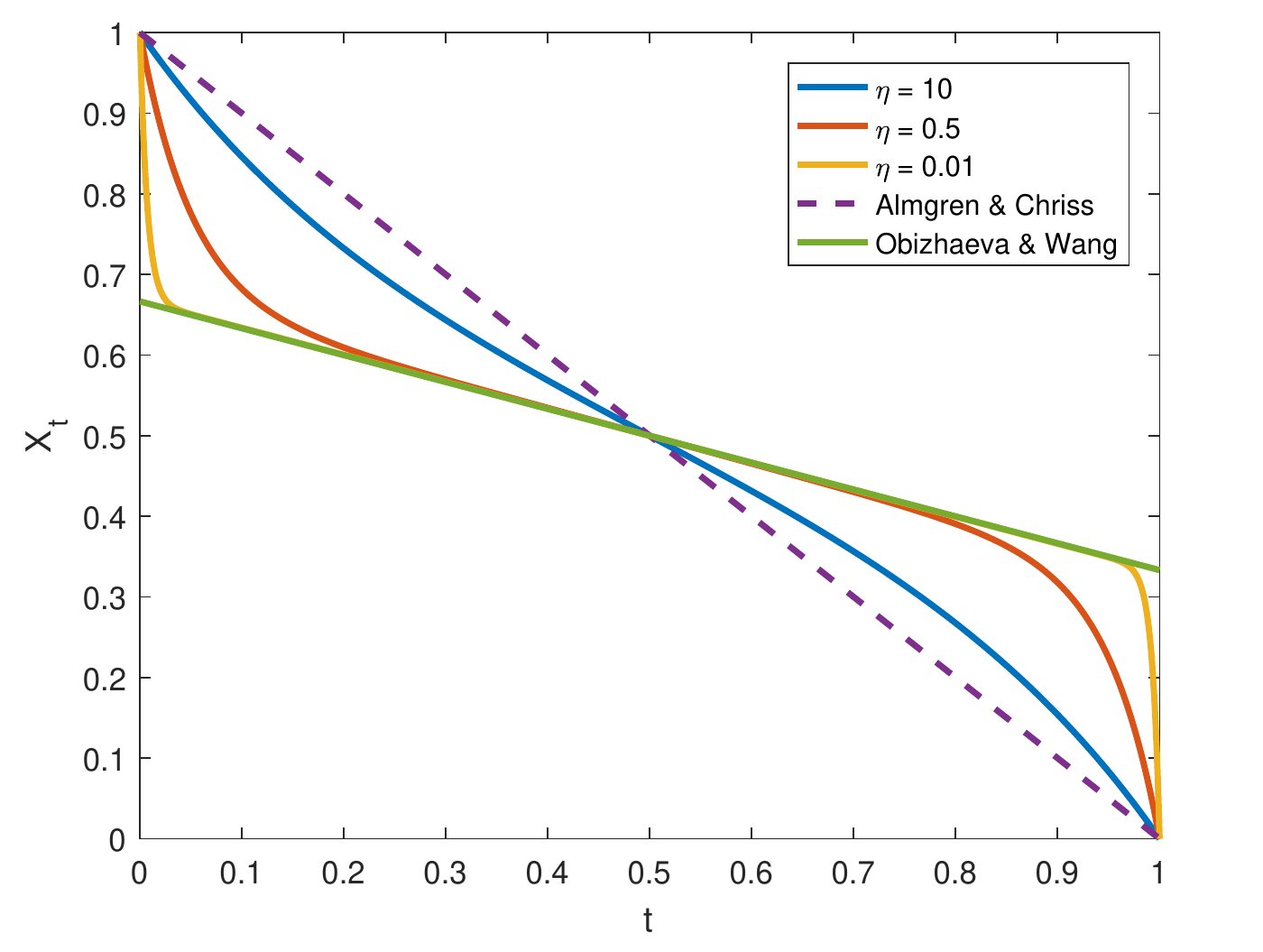}
	\caption{The optimal trading strategy in a deterministic benchmark model for different instantaneous impact factors $\eta$ compared to the models by Almgren and Chriss~\cite{AlmgrenChriss00} and by Obizhaeva and Wang~\cite{ObizhaevaWang13} for $x=1$,  $y=0$, $\lambda\equiv0$, $\gamma=100$, $T=1$, and $\rho\equiv1$.}
\end{figure}
As we see, the optimal trading strategy resembles that of the Almgre and Chriss model for large instantaneous impact factors while it resembles that of the Obizhaeva and Wang model with singular controls for small instantaneous impact factors. This suggests that our model can be viewed as a blend of the two extreme cases with only instantaneous, respectively, only persistent market impact.

\end{example}

\section{A Priori Estimates} \label{sec-a-priori}

In this section we establish a priori estimates for the BSDE system \eqref{BSDE-infty}. The estimates will be key for both, the proof of the existence of solutions and the verification theorem. Throughout, let 
\[
	((A,B,C),(Z^A,Z^B,Z^C))\in L_{\mathcal F}^\infty(\Omega;C([0,T^-];\mathbb R^3))\times L_{\mathcal F}^2(0,T^-;\mathbb R^{3\times m})
\]
denote any solution to \eqref{BSDE-infty} that satisfies \eqref{singular-terminal-condition}. 
It will be convenient to also consider the processes 
\[
	D:=\eta^{-1}(A- \gamma B)\qquad \text{and} \qquad	E:=\eta^{-1}(\gamma C-B+1)
\] 
that appear in the feedback form \eqref{feedback-form} of the candidate optimal control. The equations for $D$ and $E$ read:
\[
-dD_t=\{\eta^{-1}\lambda_t-D_t^2+\eta^{-1}\gamma\rho_tB_t-\gamma E_tD_t\}\,dt-Z^D_t\,dW_t
\]
and
\begin{align*}
	-dE_t 
	&=\{2\eta^{-1}\rho_t-2\rho_tE_t- \gamma E_t^2+\eta^{-1}\rho_tB_t-E_tD_t\}\,dt-Z_t^E\,dW_t\\ 
	 &=\{\eta^{-1}\rho_t(1-\gamma C_t)-\rho_tE_t-\gamma E_t^2-E_tD_t\}\,dt-Z_t^E\,dW_t.
\end{align*}

In order to establish the a priori estimates we first determine the range of the processes $A, \ldots, E$. The proof of the following lemma uses the multi-dimensional comparison principle for BSDEs, due to Hu and Peng~\cite{HuPeng06} presented in the Appendix. 

\begin{lemma} \label{lemma-nonnegative}
	It holds that $A,D\geq 0$ and $B,-\gamma C,\eta E\in[0,1]$, $d\mathbb P\times dt$-a.e.
\end{lemma}

\begin{proof}
	We first note that $B,C\in L^\infty_{\mathcal F}(\Omega;C([0,T^-];\mathbb R))$ together with the $L^\infty$-convergence of $B_t$ and $C_t$ as $t\rightarrow T$ implies $B,C\in L^\infty_{\mathcal F}(\Omega;C([0,T];\mathbb R))$ and hence $E\in L^\infty_{\mathcal F}(\Omega;C([0,T];\mathbb R))$.
	
	The nonpositivity of $C$ follows from the solution formula for linear BSDEs with essentially bounded coefficients \cite[Proposition~5.31]{PardouxRascanu14}. Indeed, from 
\[
\left\{
\begin{aligned}
		-d C_t &=\{-2\rho_t C_t-\eta E_t^2\}\,dt-Z^C_t\,dW_t, \quad 0\leq t<T,\\
		 C_T& = 0, 
		\end{aligned} 
	\right. 
\]	
we obtain that
\begin{equation}  \label{solution-formula-c}
C_t=-\mathbb E\left[\left.\int_t^T \eta E_s^2e^{-\int_t^s2\rho_r\,dr}\,ds\right|\mathcal F_t\right]\leq0.
\end{equation}

	The non-negativity of $E$ follows from similar arguments. In fact, 
\[ 
	-dE_t=\{-(\rho_t+\gamma E_t+D_t)E_t+\eta^{-1}\rho_t(1-\gamma C_t)\}dt-Z_t^E\,dW_t, \quad 0\leq t<T.
\]
Even though $D_t$ is singular at $t=T$, we may apply the solution formula on $[0,\tau]$ for all $\tau < T$. This yields,
\begin{equation} \label{linear-rep-E}
	E_t=\mathbb E\left[\left. E_\tau e^{-\int_t^\tau (\rho_r+\gamma E_r+D_r)\,dr}+\int_t^\tau\eta^{-1} \rho_s(1-\eta C_s)e^{-\int_t^s (\rho_r+\gamma E_r+D_r)\,dr}\,ds\right|\mathcal F_t\right].
\end{equation}
The $L^\infty$-convergence of $D_t$ to $\infty$ as $t\rightarrow T$ together with the fact that $D\in L^\infty_{\mathcal F}(\Omega;C([0,T^-];\mathbb R))$ implies that $D$ is essentially bounded below on $[0,T]$. Since $\rho, C$ and $E$ are essentially bounded we can apply the dominated convergence theorem to interchange the limit and the expectation in \eqref{linear-rep-E} when letting $\tau\rightarrow T$ in \eqref{linear-rep-E}. As a result, $E \geq 0$ because $C \leq 0$ and because 
\[
	E_\tau \to E_T=\eta^{-1}(\gamma C_T-B_T+1)=0 \quad \mbox{as } \tau \to T.
\]

	In order to prove that $B,D\geq 0$ we need we need to consider their \textit{joint} dynamics. First, due to the (improper) $L^\infty$-convergence of $B_t$ and $D_t$ as $t\rightarrow T$ there exists a deterministic time $\tau<T$ such that $B,D\geq 0$ on $[\tau,T]$. Let us consider the BSDE system for $B$ and $D$ on $[0,\tau]$:
\begin{equation*}
\left\{
\begin{aligned}
-dB_t&=\left\{-\rho_t B_t+\eta E_tD_t\right\}dt-Z_t^B\,dW_t\\
-dD_t&=\left\{\eta^{-1}\lambda_t-D_t^2+\eta^{-1}\gamma \rho_t B_t-\gamma E_tD_t\right\}dt-Z_t^D\,dW_t. \\
\end{aligned}
\right. 
\end{equation*}
Since $\rho$, $E$, and $D$ are essentially bounded on $[0,\tau]$ we may assume without loss of generality by a standard truncation argument in the $D$-variable that this system is $d\mathbb P\times dt$-a.e.~uniformly Lipschitz continuous in $B$ and $D$. Furthermore, the system is quasi-monotone because $E,\rho\geq0$. Hence, we may apply the comparison theorem for multi-dimensional BSDEs given in Proposition~\ref{prop-cp} in the Appendix with 
\begin{align*}
	f^1(t,B,D) & =(-\rho_t B+\eta E_t D, -D^2+\eta^{-1}\gamma\rho_t B-\gamma E_t D)\\
	f^2(t,B,D) &=(-\rho_t B+\eta E_t D, \eta^{-1} \lambda_t -D^2+\eta^{-1}\gamma\rho_t B-\gamma E_t D)
\end{align*}	
(up to truncation in $D$) and terminal conditions $Y^1_\tau = (0,0)$ and $Y^2_\tau = (B_\tau,D_\tau) \geq Y^1_\tau$, respectively. As the unique solution to the first BSDE system satisfies $Y^1_t \equiv (0,0)$, we see that $(B_t,D_t) = Y^2_t \geq (0,0)$ for all $t \in [0,\tau]$. Hence the process $(B,D)$ is non-negative. 
	
	Finally, we conclude from $B,-\gamma C,\eta E\geq 0$ and $\eta E=\gamma C-B+1$ that $B,-\gamma C,\eta E\leq 1$.
\end{proof}

	We are now ready to establish the a priori estimates. 
	
\begin{proposition} \label{prop-a-priori-estimates}
	In terms of $\kappa:=\sqrt{2\eta^{-1}\max\{\|\lambda\|_{L^\infty},\gamma\|\rho\|_{L^\infty}\}}$ the following a priori estimates hold $d\mathbb P\times dt$-a.e.:
\begin{equation*} 
\begin{aligned}
	\underline D_t:=\frac{\eta^{-1}\gamma}{e^{\eta^{-1}\gamma (T-t)}-1}&\leq D_t\leq \kappa \coth\left(\kappa(T-t)\right)=:\overline D_t, \\
	\underline B_t:=e^{-\|\rho\|_{L^\infty}(T-t)}&\leq B_t\leq 1,\\
	0&\leq E_t\leq \gamma^{-1} \kappa \tanh(\kappa(T-t))=:\overline E_t
\end{aligned}
\end{equation*}
\end{proposition}
\begin{proof}
	Since $D\geq 0$ we may write the BSDE for $D$ in monotone form. That is,
\begin{align*}
-dD_t&=\left\{\eta^{-1}\lambda_t-|D_t|D_t+\eta^{-1}\gamma\rho_t B_t-\gamma E_t|D_t|\right\}dt-Z_t^D\,dW_t.
\end{align*} 
The lower and upper estimate for $D$ solve
\[
-d\underline D_t=\{-|\underline D_t|\underline D_t-\eta^{-1}\gamma|\underline D_t|\}\,dt,
\]
and
\[
-d\overline D_t=\{\kappa^2 -|\overline D_t|\overline D_t \}\,dt,
\]
respectively. The preceding equations are time-homogeneous. Thus, for any $\delta>0$ the processes $\underline D_t^{\delta}:=\underline D_{t-\delta}$ and $\overline D_t^{\delta}:=\overline D_{t+\delta}$ still satisfy the respective equations but with singularities at $t=T+\delta$ and $t=T-\delta$, respectively. Since $D$ is essentially bounded on $[0,T-\delta]$ and $\lim_{t\rightarrow T-\delta} \overline D_t^{\delta}=\infty$ in $L^\infty$ there exits $s\in[0,T-\delta]$ such that $D\leq \overline D^{\delta}$ on $[s,T-\delta)$. Because $B\leq 1$ and $-ED\leq0$, we have for all $(t,y)\in[0,s]\times \mathbb R$,
\[
\eta^{-1}\lambda_t-|y|y+\eta^{-1}\gamma\rho_tB_t-\gamma E_tD_t\leq \kappa^2 -|y|y.
\] 
Hence, the classical one-dimensional comparison theorem for BSDEs with monotone drivers \cite[Proposition~5.33]{PardouxRascanu14} yields $D\leq \overline D^{\delta}$ on $[0,s]$. Finally, letting $\delta\rightarrow 0$ yields $D\leq \overline D$ on $[0,T)$ by the continuity of $\overline D$.

In order to establish $\underline D\leq D$ on $[0,T)$ one argues similarly. In this case the comparison argument is justified by the inequality  
\[
-|y|y-\eta^{-1}\gamma|y|\leq \lambda_t-|y|y+\eta^{-1}\gamma\rho_tB_t-\gamma E_t|y|.
\]

Next, we establish  the upper estimate for $E$. Since $E,D\geq 0$ we may again assume that the BSDE for $E$ is monotone, that is
\[
-dE_t=\{2\eta^{-1}\rho_t-2\rho_tE_t-\gamma |E_t|E_t-\eta^{-1}\rho_tB_t-E_tD_t\}\,dr-Z^E_t\,dW_t.
\]
Since $E,B,D\geq0$ we have for all $(t,y)\in[0,T)\times\mathbb R$ that
\begin{equation} \label{e-upper}
	2\eta^{-1}\rho_t-2\rho_t E_t-\gamma|y|y-\eta^{-1}\gamma\rho_tB_t-E_tD_t\leq \gamma^{-1}\kappa^2-\gamma |y|y.
\end{equation}
Let us consider for $\delta>0$ the deterministic process  
\[
	\overline E_t^{\delta}=\gamma^{-1}\kappa\tanh\left(\kappa(T-\delta-t)+\arctanh(\gamma \kappa^{-1}\|E_{T-\delta}\|_{L^\infty})\right), \qquad 0\leq t\leq T-\delta.
\]
Then,
\begin{equation*}
\left\{
\begin{aligned}
	-d\overline E_t^{\delta}& =\gamma^{-1}\kappa^2-\gamma|\overline E_t^{\delta}|\overline E_t^{\delta}, \qquad  0\leq t\leq T-\delta \\
	\overline E_{T-\delta}^\delta & =\|E_{T-\delta}\|_{L^\infty}.
\end{aligned}
\right.
\end{equation*}
Hence, recalling \eqref{e-upper}, the one-dimensional comparison theorem implies 
\[
	E_t\leq \overline E_t^{\delta}, \qquad t\in[0,T-\delta].
\]
Since $\|E_{T}\|_{L^\infty}=0$, letting $\delta\rightarrow 0$ completes the proof.

	Finally, to establish the lower estimate for $B$ one notices that $\underline B$ solves
\[
	-d\underline B_t=-\|\rho\|_{L^\infty}\underline B_t\,dt, \quad 0\leq t\leq T; \qquad \underline B_T=1,
\]
and is hence a subsolution to the BSDE for $B$. At this point, we already know that the potential singular term $E_tD_t$ in the BSDE for $B$ behaves well (being bounded by $\overline E_t\overline D_t=\gamma^{-1}\kappa^2$) on the entire interval $[0,T]$. Hence, no shifting argument at the terminal time is needed in this step and we conclude directly by comparison that $\underline B\leq B$. 
\end{proof}

	From the a priori estimates we obtain the asymptotic behavior of our BSDE system at the terminal time as stated in the following corollary. The asymptotic at the terminal time is key to our existence result. 
\begin{corollary} \label{cor-asymptotics}
	The following asymptotic behaviors hold in $L^\infty$ as $t\rightarrow T$:
\begin{equation*}
\begin{aligned}
	(T-t)A_t&=\eta+O(T-t),\\
B_t&=1+O(T-t),\\ 
	C_t&=O((T-t)^3).
\end{aligned}
\end{equation*}
\end{corollary}
\begin{proof} 
	The asymptotic behavior of $A=\eta(D+\gamma B)$ and $B$ follows directly from the a priori estimates given above. The asymptotic order of $C$ follows from~\eqref{solution-formula-c} and $E_t=O(\overline E_t)=O(T-t)$ in $L^\infty$ as $t\rightarrow T$.
\end{proof}

\section{Existence} \label{sec-existence}

	In this section we prove Theorem~\ref{thm-existence}, i.e. the existence of a solution to the BSDE syetem~\eqref{BSDE-infty} that satisfies the singular terminal condition \eqref{singular-terminal-condition}. Similarly as in~\cite{GraeweHorstSere13}, our proof of existence is based on the asymptotic behavior established in Corollary~\ref{cor-asymptotics}. It suggests the following asymptotic ansatz:
\begin{equation} \label{asymptotic-ansatz}
\begin{aligned}
	A_t&=\frac{\eta}{T-t}+\frac{H_t}{(T-t)^2}, && H_t=O((T-t)^2) \text{ in $L^\infty$ as $t\rightarrow T$},\\
  B_t&=1+\frac{G_t}{T-t}, & & G_t=O((T-t)^2) \text{ in $L^\infty$ as $t\rightarrow T$},\\
 	C_t&=P_t, && P_t=O((T-t)^2) \text{ in $L^\infty$ as $t\rightarrow T$},
\end{aligned}
\end{equation}
where the asymptotic order of $H$ and $G$ is raised artificially for similar reasons as in~\cite[Remark~4.2]{GraeweHorstSere13} to obtain the locally Lipschitz type statement given in Lemma~\ref{lemma-fixed-point}(ii) below, while the reduced order of $P$ unifies the notation and allows us to solve for all three processes in the same weighted $L^\infty$-space.   

	The asymptotic ansatz~\eqref{asymptotic-ansatz} reduces the original system \eqref{BSDE-infty} to 
\begin{equation} \label{BSDE_0} 
\!\left\{
\begin{aligned}
	\!-dH_t&= \left\{(T-t)^2\lambda_t-\frac{1}{\eta}\left(\frac{H_t}{T-t}-\gamma (T-t+G_t)\right)^2+2\gamma(T-t+G_t)\right\}\!dt+Z_t^HdW_t\\
	\!-dG_t& =\left\{\!-\rho_t(T\!-\!t+G_t)\!+\!\frac{1}{\eta}\!\left(\!\gamma P_t-\frac{G_t}{T\!-\!t}\right)\!\!\left(\frac{H_t}{T\!-\!t}\!-\!\gamma(T\!-\!t+G_t)\!\right)\!+\!\gamma P_t\right\}\!dt+Z_t^GdW_t\\
	\!-dP_t&= \left\{-2\rho_t P_t-\frac{1}{\eta}\left(\gamma P_t-\frac{G_t}{T-t}\right)^2\right\}\!dt+Z_t^PdW_t.
\end{aligned}
\right.
\end{equation}
We define $f:\Omega\times[0,T)\times\mathbb R^3\rightarrow \mathbb R^3$ such that we have 
\[
-dY_t=f(t,Y_t)\,dt-Z_t\,dW_t
\]
as a compact notation for~\eqref{BSDE_0} by identifying $Y=(H,G,P)$ and $Z=(Z^H,Z^G,Z^P)$. For $\delta>0$ specified below, we will establish the existence of a short-time solution to~\eqref{BSDE_0} in the space 
\[
	\mathcal H=\{Y\in L_{\mathcal F}^\infty(\Omega;C([T-\delta,T];\mathbb R^3)):\|Y\|_{\mathcal H}<+\infty\}
\] 
endowed with the norm
\[
\|Y\|_{\mathcal H}=\left\|(T-\cdot)^{-2}Y_\cdot\right\|_{L^\infty_{\mathcal F}(\Omega;C([T-\delta,T];\mathbb R^3))}.
\]
Since $Y_T = 0$ this means that we are looking for a fixed point in $\mathcal H$ of the operator
\[
	\Gamma(Y):=\bigg(\mathbb E\left[\left.\int_t^Tf(s,Y_s)\,ds\right|\mathcal F_t\right]\bigg)_{T-\delta\leq t\leq T}.
\]

\begin{lemma} \label{lemma-fixed-point}
	The following holds:
\begin{itemize}
	\item[(i)] $\mathcal H$ is complete.
	\item[(ii)] For every $R>0$ there exists a constant $L>0$ (independent of $\delta$) such that
	\[\| f(\cdot, Y_\cdot)-f(\cdot, X_\cdot)\|_{\mathcal H}\leq L\| Y_\cdot- X_\cdot\|_{\mathcal H}\quad \mbox{for all } Y, X\in\overline B_{\mathcal H}(R).\]
\end{itemize}
\end{lemma}
\begin{proof}
	The spaces $L_{\mathcal F}^\infty(\Omega;C([T-\delta,T];\mathbb R))$ and $\mathcal H$ are isometrically isomorphic by identifying $Y\in L_{\mathcal F}^\infty(\Omega;C([T-\delta,T];\mathbb R))$ with the process $((T-t)^2Y_t)_{T-\delta\leq t\leq T}$ in $\mathcal H$. Hence $\mathcal H$ is complete. 

	In order to establish the Lipschitz continuity, let $\overline{ Y_tY_t'}$ be the line segment connecting $Y_t$ and~$Y_t'$. By the mean value theorem we have for $Y,Y'\in \overline B_{\mathcal H}(R)$, $d\mathbb P\times dt$-a.e.,
\begin{align}
	|f(t,Y_t)-f(t,Y_t')|&\leq \sup_{y\in \overline{Y_tY_t'}}\left\|\partial_y f(t,y)\right\|_{\textup{Hom}(\mathbb R^3;\mathbb R^3)}|Y_t-Y_t'| \nonumber\\
	&\leq (T-t)^2\sup_{|y|\leq (T-t)^2 R}\left\|\partial_y f(t,y)\right\|_{\textup{Hom}(\mathbb R^3;\mathbb R^3)}\|Y-Y\|_{\mathcal H}, \label{time-weighted-Lip} 
\end{align}
where it is used that the line $\overline{ Y_tY_t'}$ is contained in $\overline B_{\mathbb R^3}((T-t)^2R)$, $d\mathbb P\times dt$-a.e. But,
\[
\partial_y f(t,y)=
\begin{pmatrix} 
	\frac{-2y_1}{\eta(T-t)^2}+\frac{2\gamma y_2}{\eta (T-t)}+\frac{2\gamma}{\eta} & \frac{2\gamma y_1}{\eta(T-t)}-\frac{2(y_2+T-t-\eta\gamma^{-1})}{\eta\gamma^{-2}} & 0 \\ 
	\frac{-y_2}{\eta (T-t)^2}+\frac{\gamma y_3}{\eta(T-t)} & \frac{-y_1}{\eta(T-t)^2}+\frac{2\gamma y_2}{\eta(T-t)}-\frac{\gamma y_3-1}{\eta\gamma^{-1}}-\rho_t & \frac{\gamma y_1}{\eta(T-t)}-\frac{y_2+T-t-1}{\eta\gamma^{-2}}\\
	0 & \frac{-2y_2}{\eta(T-t)^2} +\frac{2\gamma y_3}{\eta(T-t)} & \frac{2\gamma y_2}{\eta(T-t)}-\frac{2 y_3}{\eta\gamma^{-2}}-2\rho_t
\end{pmatrix},
\]
from which we see that the supremum in~\eqref{time-weighted-Lip} is essentially bounded on $\Omega\times[T-\delta,T]$.
%
\end{proof}

Choosing $R$ and $\delta$ appropriately the preceding lemma allows us to use a standard fix-point argument to show that $\Gamma$ has a unique fix-point. The fix-point is just a local solution to \eqref{BSDE_0}.

\begin{proposition} \label{prop-local-existence}
	For $\delta>0$ sufficient small there exists a short-time solution $(Y,Z)\in \mathcal H\times L^2_{\mathcal F}(T-\delta,T;\mathbb R^{3\times m})$ to~\eqref{BSDE_0}.
\end{proposition}

\begin{proof}
Let us fix $R=4\max\{T\|\lambda\|_{L^\infty}+\gamma^2 T/\eta+2\gamma,\|\rho\|_{L^\infty}\}$ and choose $L>0$ as in Lemma~\ref{lemma-fixed-point}. For $Y,Y'\in\overline B_{\mathcal H}(R)$ it then holds $d\mathbb P\times dt$-a.e.,
\begin{align*}
	|\Gamma(Y)_t-\Gamma(Y')_t|&\leq \mathbb E\left[\left.\int_t^T|f(s,Y_s)-f(s,Y'_s)|\,ds\right|\mathcal F_t\right]\\
	&\leq L(T-t)^3\|Y-Y'\|_{\mathcal H}
\end{align*}
This yields, as long as $0<\delta\leq (2L)^{-1}$,
\[
\|\Gamma(Y)-\Gamma(Y')\|_{\mathcal H}\leq \frac{1}{2}\|Y-Y'\|_{\mathcal H}.
\]
Hence, $\Gamma$ is an $1/2$-contraction on $\overline B_{\mathcal H}(R)$. Furthermore, $\Gamma$ maps $\overline B_{\mathcal H}(R)$ onto itself. Indeed, for all $Y\in\overline B_{\mathcal H}(R)$ it holds $d\mathbb P\times dt$-a.e.,
\begin{align*}
	|\Gamma(Y)_t|&\leq |\Gamma(Y)_t-\Gamma(0)_t|+|\Gamma(0)_t|\\
	&\leq (T-t)^2\frac{R}{2}+\mathbb E\left[\left.\int_t^T|f(s,0)|\,ds\right|\mathcal F_t\right]\\
&\leq (T-t)^2\frac{R}{2}+\mathbb E\left[\left.\int_t^T2\max\{(T-s)^2\lambda_s+\frac{\gamma^2}{\eta}(T-s)^2+2\gamma(T-s),\rho_s(T-s)\}\,ds\right|\mathcal F_t\right]\\	
	&\leq (T-t)^2\frac{R}{2}+ 2(T-t)^2\max\{T\|\lambda\|_{L^\infty}+\frac{\gamma^2}{\eta} T+2\gamma,\|\rho\|_{L^\infty}\}=(T-t)^2R.
\end{align*}
As a result, $\Gamma$ has a unique fixed point $Y\in\overline B_{\mathcal H}(R)$. The process $Y$ satisfies
\[
	Y_t= -\int_{T-\delta}^tf(s,Y_s)\,ds+\mathbb E\left[\left.\int_{T-\delta}^Tf(s,Y_s)\,ds\right|\mathcal F_t\right].
\]
By the martingale representation theorem there exits a process $Z\in L^2_{\mathcal F}(T-\delta,T;\mathbb R^{3\times m})$ such that
\[
	Y_t= -\int_{T-\delta}^tf(s,Y_s)\,ds+\int_{T-\delta}^tZ_s\,dW_s.
\]
Hence, $(Y,Z)$ gives the desired short-time solution to~\eqref{BSDE_0}.
\end{proof}

We are now ready to prove Theorem~\ref{thm-existence}.

\begin{proof}[Proof of Theorem~\ref{thm-existence}]
	The short-time solution to~\eqref{BSDE_0} established by Proposition~\ref{prop-local-existence} gives in terms of the ansatz~\eqref{asymptotic-ansatz} a short-time solution 
\[
(A,B,C)\in	L^\infty_{\mathcal F}(\Omega;([T-\delta,T^-];\mathbb R^3))\times L^2_{\mathcal F}(T-\delta,T^-;\mathbb R^{3\times m})
\] 
to~\eqref{BSDE-infty} that satisfies the singular terminal condition~\eqref{singular-terminal-condition}. In order to see that this short-time solution extends to a global solution in $L^\infty_{\mathcal F}(\Omega;C([0,T^-];\mathbb R^3))\times L^2_{\mathcal F}(0,T^-;\mathbb R^{3\times m})$ notice first that the system~\eqref{BSDE-infty} satisfies the assumptions of the local $L^\infty$-existence results for BSDEs with locally Lipschitz drivers of Lemma~\ref{lemma-local-lip-l-infty-existence} given in the appendix. Hence, the system~\eqref{BSDE-infty} imposed with the essentially bounded terminal value $(A_{T-\delta},B_{T-\delta},C_{T-\delta})$ admits an essentially bounded local extension on $[T-\delta-\delta',T-\delta]$. 
	 
	 Due to the a priori estimates given in Proposition~\ref{prop-a-priori-estimates} we know that this local extension will stay (recalling $A=\eta(D+\gamma B)$) in the bounded region $[0,\eta(\overline D_{T-\delta}+\gamma)]\times[0,1]\times[-1/\gamma,0]$. When iterating this extension procedure we may therefore choose (cf.\ the proof of Lemma~\ref{lemma-local-lip-l-infty-existence}) step by step the same local Lipschitz constant $L>0$ for the system~\eqref{BSDE-infty}, which results in a constant length $\delta'>0$ of the extension interval. Thus, after finitely many steps we obtain a global extension on $[0,T)$.
\end{proof}

\section{Verification} \label{sec-verification}

	This section devoted to the verification statement of Theorem~\ref{thm-verification}. Throughout, let 
\[
	((A,B,C),(Z^A,Z^B,Z^C))\in L_{\mathcal F}^\infty(0,T^-;\mathbb R^3)\times L_{\mathcal F}^2(0,T^-;\mathbb R^{3\times m})
\]
denote any solution to~\eqref{BSDE-infty} that satisfies~\eqref{singular-terminal-condition} and recall that the candidate optimal strategy~$\xi^*$  is given in terms of the processes
\[
	D:=\eta^{-1}(A-\gamma B) \qquad \text{and} \qquad E:=\eta^{-1}(\gamma C-B+1)
\]
for which a priori estimates have been established in Section~\ref{sec-a-priori}. The proof of the admissibility of $\xi^*$ uses the following iterated integral version of Gronwall's inequality.

\begin{lemma}[{\cite[Corollary~11.1]{BainovSimeonov92}}]  \label{lemma-iterated-gronwall}
	Let $u(t)$, $a(t)$, and $b(t)$ be nonnegative continuous functions on $[0,T]$ with $a(t)$ and $b(t)$ being nondecreasing, and suppose
	\[u(t)\leq a(t)+b(t)\int_0^t\int_0^sk(s,r)u(r)\,dr\,ds, \qquad 0\leq t\leq T,\]
	where $k(s,r)$ is a nonnegative continuous function on $\{0\leq r\leq s\leq T\}$. Then 
	\[u(t)\leq a(t)\exp\left(b(t)\int_0^t\int_0^s k(s,r)\,dr\,ds\right), \qquad 0\leq t\leq T.\]
\end{lemma}

We are now ready to verify that the candidate optimal control $\xi^*$ is indeed admissible. 

\begin{lemma} \label{lemma-admissible}
	The feedback control $\xi^*$ given in \eqref{feedback-form} is admissible.
\end{lemma}
\begin{proof}
	Let us fix an initial state $(t,x,y)\in[0,T)\times\mathbb R\times \mathbb R$. The dynamics of the state process $(X^*,Y^*)$ under the candidate optimal control $\xi^*$ is given by:
\begin{equation} \label{optimal-state-dynamics}
\left\{
\begin{aligned}
dX_s^*&=\{-D_sX_s^*+E_sY_s^*\}\,ds\\
dY_s^*&=\{-(\rho_s+\gamma E_s)Y_s^*+\gamma D_sX_s^*\}\,ds.
\end{aligned}
\right.
\end{equation}
Due to the singularity of $D$ at the terminal time, it is not clear yet that the solution to~\eqref{optimal-state-dynamics} is well-defined at the terminal time;  a priori we only know that $(X^*,Y^*)\in L^\infty_{\mathcal F}(\Omega;C([t,T^-];\mathbb R^2)$. 

In order to show that $X^*_T=0$ we first apply the variation of constants formula for $t\leq s<T$ to get:
\begin{align*}
X_s^*=xe^{-\int_t^sD_u\,du}+\int_t^se^{-\int_r^s D_u\,du}E_rY_r^*\,dr
\end{align*}
and
\begin{equation} \label{voc-rep-y-optimal}
		Y_s^* =ye^{-\int_t^s(\rho_u+\gamma E_u)\,du}+\int_t^se^{-\int_r^s(\rho_u+\gamma E_u)\,du}\gamma D_rX_r^*\,dr.
\end{equation}
Hence, the process $\tilde X_s:=X_s^*e^{\int_t^sD_r\,dr}$ satisfies,
\begin{align*}
	\tilde X_s&= x+\int_t^se^{\int_t^r D_u\,du}E_rY_r^*\,dr\\
	&=x+\int_t^se^{\int_t^r D_u\,du}E_r\left(ye^{-\int_t^r(\rho_u+\gamma E_u)\,du}+\int_t^re^{-\int_u^r(\rho_v+\gamma E_v)\,dv}\gamma D_u e^{-\int_t^uD_v\,dv}\tilde X_u\,du\right)dr.
\end{align*}
Since $\rho,E\geq 0$, this yields,
\begin{align*}
	|\tilde X_s|&\leq|x|+\int_t^se^{\int_t^r D_u\,du}E_r|y|\,dr+\int_t^se^{\int_t^r D_u\,du}E_r\int_t^r\gamma D_u e^{-\int_t^uD_v\,dv}|\tilde X_u|\,du\,dr\\
	&=|x|+|y|\int_t^se^{\int_t^r D_u\,du}E_r\,dr+\int_t^s\gamma E_r\int_t^rD_u e^{\int_u^rD_v\,dv}|\tilde X_u|\,du\,dr.
\end{align*}
By the iterated integral version of Gronwall's inequality (Lemma~\ref{lemma-iterated-gronwall}),
\begin{equation} \label{gronwall-inequality}
\begin{aligned}
	|\tilde X_s|&\leq \left(|x|+|y|\int_t^se^{\int_t^r D_u\,du}E_r\,dr\right)\exp\left(\int_t^s\gamma E_r\int_t^rD_u e^{\int_u^rD_v\,dv}\,du\,dr\right) \\
	&=\left(|x|+|y|\int_t^se^{\int_t^r D_u\,du}E_r\,dr\right)\exp\left(\int_t^s\gamma E_r\left(e^{\int_t^r D_u\,du}-1\right)dr\right).
\end{aligned}
\end{equation}
In view of the a priori upper bounds on $D$ and $E$, because the antiderivative of $\coth(\cdot)$ is given by $\ln(\sinh(\cdot))$ and because $\cosh(\cdot)\geq 1$,
\begin{align*}
	\int_t^s\gamma E_re^{\int_t^r D_u\,du}\,dr &\leq \int_t^s\kappa\tanh(\kappa(T-r))e^{\int_t^r \kappa \coth(\kappa(T-u))\,du}\,dr\\
	&=\int_t^s\kappa \tanh(\kappa(T-r))\frac{\sinh(\kappa(T-t))}{\sinh(\kappa(T-r))}\,dr\\
	&\leq \kappa(s-t)\sinh(\kappa(T-t))\\
	&\leq \kappa T\sinh(\kappa T).
\end{align*}
Along with \eqref{gronwall-inequality} this shows that $|\tilde X_s|$ is bounded as $s\rightarrow T$. Therefore, this time using the a priori lower bound for $D$, 
\begin{equation} \label{optimal-x-state-bound}
\begin{aligned}
	|X_s^*|&= |\tilde X_s|\exp\left(-\int_t^sD_r\,dr\right)\\
	&\leq |\tilde X_s|\exp\left(-\int_t^s \frac{\eta^{-1}\gamma}{e^{\eta^{-1}\gamma(T-r)}-1}\,dr\right)\\
	&=|\tilde X_s|\frac{1-e^{-\eta^{-1}\gamma(T-s)}}{1-e^{-\eta^{-1}\gamma(T-t)}}\xrightarrow{s\rightarrow T} 0.
\end{aligned}
\end{equation}
This shows that $X^*_T=0$. It also shows that $X^*_s=O(T-s)$ in $L^\infty$ as $s\rightarrow T$. As $D_s=O((T-s)^{-1})$ it follows that $$DX^*\in L^\infty_{\mathcal F}(\Omega;C([t,T];\mathbb R)).$$ The boundedness of $DX^*$ again implies by \eqref{voc-rep-y-optimal} that $Y^*\in L^\infty_{\mathcal F}(\Omega;C([t,T];\mathbb R))$. Hence, we conclude $$\xi^*\in L^\infty_{\mathcal F}(\Omega;C([t,T];\mathbb R)).$$
This proves that $\xi^*$ is indeed admissible.
\end{proof}


\begin{lemma} \label{lemma-vanishing}
	For every $\xi\in\mathcal A(t,x)$ it holds $ E[A_s X_s^2+B_sX_sY_s+C_sY_s^2|\mathcal F_t]\xrightarrow{s\rightarrow T}0.$
\end{lemma}
\begin{proof} Recalling $B,C\in L_{\mathcal F}^\infty(\Omega;C([0,T];\mathbb R))$, $X,Y\in L^2_\mathcal F(\Omega;C([t,T];\mathbb R))$, and $X_T=C_T=0$, it follows by the dominated convergence theorem,
\[ 
	\mathbb E[B_sX_sY_s+C_sY_s^2|\mathcal F_t]\xrightarrow{s\rightarrow T}0.
\]
Furthermore, note that by $X_T=0$ and Jensen's inequality,
\[
	X_s^2=\left(\int_s^T\xi_r\,dr\right)^2\leq (T-s)\int_s^T\xi_r^2\,dr.
\]
Hence, by Corollary~\ref{cor-asymptotics},
\[
	\mathbb E[A_sX_s^2|\mathcal F_t]\leq \mathbb E\left[\left.(T-s)A_s\int_s^T\xi_r^2\,dr\right|\mathcal F_t\right]\xrightarrow{s\rightarrow T} 0. \qedhere
\]
\end{proof}

	We are now ready to prove the verification theorem.
	
\begin{proof}[Proof of Theorem~\ref{thm-verification}] 
	By a slight abuse of notation we \textit{define} within this proof the random fields $V_t(x,y)$ and $Z_t(x,y)$ by the linear-quadratic ansatz~\eqref{LQ-ansatz} and verify that this gives indeed the value function of the control problem. For the moment we only know that $(V,Z)$ is a classical solution the HJB equation~\eqref{HJB}.

	Let us fix an initial state $(t,x,y)\in [0,T)\times \mathbb R\times\mathbb R$ and admissible control $\xi\in\mathcal A(t,x)$. For $n\in\mathbb N$ we define the stopping time
\[
\tau_n:=\inf\{t\leq s\leq T:|X_s|\vee|Y_s|\geq n\}.
\]

Since $(V,Z)$ solve the HJB equation, it holds by the It\^o-Kunita formula~\cite[Theorem~I.8.1]{Kunita84} for all $t\leq s<T$, 
\begin{equation} \label{ito-kunita}
\begin{aligned}
	V_t(x,y) &=  V_{s\wedge\tau_n}(X_{s\wedge\tau_n},Y_{s\wedge\tau_n})+\int_t^{s\wedge\tau_n}\{\xi_r\partial_xV_r(X_r,Y_r)-(-\rho_rY_r+\gamma \xi_r)\partial_yV_r(X_r,Y_r)\}\,dr\\
	&\quad+\int_t^{s\wedge\tau_n}\inf_{\xi\in\mathbb R}\{-\xi\partial_x V_r(X_r,Y_r)-(\rho_t Y_r-\gamma\xi)\partial_y V_r(X_r,Y_r)+\tfrac{1}{2}\eta\xi^2+\xi y+\tfrac{1}{2}\lambda_t x^2\}\,dr\\
	&\quad-\int_t^{s\wedge\tau_n}Z_r(X_r,Y_r)\,dW_r.
\end{aligned}
\end{equation}
The above stochastic integral stopped at $\tau_n$ is a true martingale. Hence,
\begin{equation} \label{ito-kunita-in-expectation}
	V_t(x,y)\leq \mathbb E[ V_{s\wedge\tau_n}(X_{s\wedge\tau_n},Y_{s\wedge\tau_n})|\mathcal F_t]+\mathbb E\left[\left.\int_t^{s\wedge\tau_n}\{\tfrac{1}{2}\eta\xi_r^2+\xi_rY_r+\tfrac{1}{2}\lambda_rX_r^2\}\,dr\right|\mathcal F_t\right].
\end{equation}
Since the coefficients $A,B,C$ of the random field $V_r(x,y)$ are essentially bounded on $[t,s]$, since $X,Y\in L^2_{\mathcal F}(\Omega;C([t,T];\mathbb R))$, and because $\xi\in L^2_{\mathcal F}(t,T;\mathbb R)$ and $\lambda\in L^\infty_{\mathcal F}(0,T;\mathbb R_+)$, it follows by H\"older's inequality that 
\[
	AX^2,BXY,CY^2\in L^1_{\mathcal F}(\Omega;C([t,s];\mathbb R)) \quad \mbox{and} \quad 
	\xi^2,\xi Y,\lambda X^2\in L^1_{\mathcal F}(t,T;\mathbb R). 
\]	
Hence, the dominated convergence theorem applies when letting $n\rightarrow \infty$ in~\eqref{ito-kunita-in-expectation}, which yields,
\begin{equation}
	V_t(x,y)\leq \mathbb E[ V_{s}(X_{s},Y_{s})|\mathcal F_t]+\mathbb E\left[\left.\int_t^{s}\{\tfrac{1}{2}\eta\xi_r^2+\xi_rY_r+\tfrac{1}{2}\lambda_rX_r^2\}\,dr\right|\mathcal F_t\right].
\end{equation}
Hence, by Lemma~\ref{lemma-vanishing} and again the dominated convergence theorem letting $s\rightarrow T$ yields,
\begin{equation} \label{suboptimal}
V_t(x,y)\leq \mathbb E\left[\left.\int_t^T\{\tfrac{1}{2}\eta\xi_r^2+\xi_rY_r+\tfrac{1}{2}\lambda_rX_r^2\}\,dr\right|\mathcal F_t\right].
\end{equation}
Finally note that since the feedback control $\xi^*$ attains the infimum in~\eqref{ito-kunita} it holds equality in \eqref{ito-kunita-in-expectation}--\eqref{suboptimal} if $\xi=\xi^*.$ 
\end{proof}

\section{Conclusion}

In this paper we analyzed a novel stochastic optimal control problem arising in models of optimal trade execution with instantaneous and persistent price impact and stochastic resilience. Assuming that the instantaneous impact factor is constant but allowing for stochastic resilience and market risk we characterized the value function in terms of the unique solution to a three-dimensional stochastic Riccati equation with singular terminal condition in the first component. Our existence of solutions results used an extension of the asymptotic expansion approach introduced in \cite{GraeweHorstSere13} to a multi-dimensional setting. Several open problems remain. First, we cannot guarantee non-negativity of the trading rate. Intuitively, price-triggered round trips should not be beneficial if $y=0$. Based on our analysis, they can not be ruled out, though. Second, the assumption that $\eta$ and $\gamma$ are constant was important to establish the a priori estimates. An extension to more general impact factors, especially a random impact factor~$\gamma$ is certainly desirable as suggested in \cite{FruthSchoenebornUrusov15}. Third, a numerical analysis of a deterministic benchmark model suggests that our model can be viewed as a approximation to a model with both absolutely continuous and singular controls if $\eta \to 0$. While a formal proof of this limit result in a general non-Markovian framework would certainly be desirable it is clearly beyond the scope of the present paper.

\appendix
\section{Appendix} \label{sec-appendix}

	A necessary and sufficient condition under which the comparison theorem holds for multi-dimensional BSDEs has been first given by Hu and Peng~\cite{HuPeng06}. The equivalent quasi-monotonicity condition (iv) below can be found in~\cite[Theorem~3.1]{Xu16}. The comparison results in \cite{HuPeng06, Xu16} are stated under an additional continuity condition on the drivers that is not satisfied in our model. However, the continuity condition is only needed to prove that if a comparison principle holds, then the system is necessarily quasi-monotone. Continuity is not needed for the converse implication. As such, their results are in fact applicable to our framework. Even though, for the reader's convenience we refer instead to a comparison result for multi-dimensional reflected BSDEs by Wu and Xiao~\cite{WuXiao10} that is formulated explicitly under the weaker regularity assumption~(i) given below.
	
\begin{proposition}[{\cite[Theorem~3.1]{WuXiao10}}] \label{prop-cp}
	Let $(Y^i,Z^i)\in L^2_{\mathcal F}(\Omega;C([0,T];\mathbb R^d))\times L^2_{\mathcal F}(0,T;\mathbb R^{d\times m})$, $i=1,2$, be solutions to the BSDEs
\[
-dY_t^i=f^i(t,Y_t^i,Z_t^i)\,dt-Z_t^i\,dW_t, \qquad 0\leq t\leq T,
\]
with the drivers $f^i:\Omega\times[0,T]\times\mathbb R^d\times\mathbb R^{d\times m}\rightarrow \mathbb R^d$, $i=1,2$, satisfying 
\begin{itemize}
	\item[(i)] $f^i(\cdot,y,z)\in L^2_{\mathcal F}(0,T;\mathbb R^d)$ for all $y\in\mathbb R$ and $z\in\mathbb R^{d\times m}$,
	\item[(ii)] there exits $L>0$ such that for all $y,y'\in\mathbb R^d$ and $z,z'\in\mathbb R^{d\times m}$,
	\[|f^i(t,y,z)-f^i(t,y',z')|\leq L(|y-y'|+|z-z'|),\qquad d\mathbb P\times dt\text{-a.e.,}
	\]
\end{itemize}
and suppose, in addition,
\begin{itemize}
	\item[(iii)] $Y_T^1\leq Y_T^2$,
	\item[(iv)] for every $k=1,\ldots,d$ it holds for all $y^1,y^2\in\mathbb R^d$ and $z^1,z^2\in\mathbb R^{d\times m}$ such that $y_k^1=y_k^2$, $z_k^1=z_k^2$, $y_l^1\leq y_l^2$, $l\neq k$:
\[
f_k^1(t,y^1,z^1)\leq f^2_k(t,y^2,z^2), \qquad \text{ $d\mathbb P\times dt$-a.e.}
\]
\end{itemize}
Then $Y^1_t\leq Y^2_t$,  $t\in[0,T]$.
\end{proposition}

	Below we state a local $L^\infty$-existence result for BSDEs with locally Lipschitz drivers not depending on $Z$. The result seems well well-known; we give it for completeness. Specifically, we consider the BSDE
\begin{equation} \label{L-infty-BSDE}
	Y_t= \zeta+\int_t^Tf(s,Y_s)\,ds-\int_t^TZ_s\,dW_s, \qquad 0\leq t\leq T,
\end{equation}
where we assume that the terminal value
\begin{itemize}
	\item $\zeta\in L^{\infty}_{\mathcal F_T}(\mathbb R^d)$ 
\end{itemize}
is essentially bounded and that the driver  $f:\Omega\times[0,T]\times\mathbb R^d\rightarrow \mathbb R^d$ satisfies
			\begin{itemize}
				\item $f(\cdot,0)\in L^{\infty}_{\mathcal F}(0,T;\mathbb R^d)$,
				\item for every $R>0$ there exists $L>0$ such that for all $|y|,|y'|\leq R$,
					\begin{equation} \label{locally-lip}
						|f(t,y)-f(t,y')|\leq L |y-y'|.
					\end{equation}
			\end{itemize}

\begin{lemma} \label{lemma-local-lip-l-infty-existence}
	Under the above assumptions there exits $\delta>0$ such that there exits on $[T-\delta,T]$ a short-time solution $(Y,Z)\in L_{\mathcal F}^\infty(\Omega;C([T-\delta,T];\mathbb R^d))\times L^2_{\mathcal F}(T-\delta,T;\mathbb R^{d\times m})$ to \eqref{L-infty-BSDE}.
\end{lemma}

\begin{proof}
	We will show that one may choose $\delta=1/(2L)$, where $L$ is the Lipschitz constant given in~\eqref{locally-lip} with respect to $R=2(\|\zeta\|_{L^\infty}+T \|f(\cdot,0)\|_{L^\infty})$.
	
	With $\mathcal H=L_{\mathcal F}^\infty(\Omega;C([0,T];\mathbb R^d))$ we define the operator $\Gamma:\mathcal H\rightarrow \mathcal H$ by
	\[
	\Gamma(Y)_t =\mathbb E\left[\left.\zeta+\int_t^Tf(s,Y_s)\,ds\right |\mathcal F_t\right].
	\]
Then $\Gamma$ is a contraction on $\overline B_{\mathcal H}(R)$: For all $Y,Y'\in\overline B_{\mathcal H}(R)$ it holds $d\mathbb P\times dt$-a.e.,
\begin{align*}
	|\Gamma(Y)_t-\Gamma(Y'_t)| &\leq \mathbb E\left[\left.\int_t^T|f(s,Y_s)-f(s,Y_s')| \,ds\right|\mathcal F_t\right]\\
	&\leq L \mathbb E\left[\left.\int_t^T|Y_s-Y_s'|\,ds\right|\mathcal F_t\right]\\
	&\leq L(T-t)\|Y-Y'\|_{L^\infty}\\
	&\leq L\delta \|Y-Y'\|_{L^\infty}=\frac{1}{2} \|Y-Y'\|_{L^\infty}.
\end{align*}
Furthermore, $\Gamma$ maps $\overline B_{\mathcal H}(R)$ into itself: For all $Y\in\overline B_{\mathcal H}(R)$ it holds $d\mathbb P\times dt$-a.e.,
\begin{align*}
	|\Gamma(Y)_t|&\leq |\Gamma(Y)_t-\Gamma(0)_t|+|\Gamma(0)_t|\\
		&\leq \|\Gamma(Y)-\Gamma(0)\|_{L^\infty}+\|\zeta\|_{L^\infty}+(T-t)\|f(\cdot,0)\|_{L^\infty}\\
		&\leq \frac{1}{2}\|Y\|_{L^\infty}+\|\zeta\|_{L^\infty}+T \|f(\cdot,0)\|_{L^\infty}\leq R.
\end{align*}
Hence, $\Gamma$ has a unique fixed point in $\overline B_{\mathcal H}(R)$. By the martingale representation theorem, this fixed point gives the desired solution.
\end{proof}

\bibliographystyle{siam}
\bibliography{gh}

\end{document}